\documentclass[11pt]{article}
\usepackage{amsmath}
\usepackage{amssymb}
\usepackage{latexsym}
\usepackage{amsthm}
\usepackage{mathrsfs} 
\usepackage{wasysym}
\usepackage{fancyhdr}
\usepackage{xcolor}
\usepackage{amsfonts}
\usepackage{amssymb}
\usepackage{epsfig}
\usepackage{epstopdf}
\usepackage[normalem]{ulem}
\usepackage[T1]{fontenc}
\usepackage{eucal}
\usepackage{tikz}
\usepackage[utf8]{inputenc}
\usepackage{hyperref}
\usepackage{algorithm}
\usepackage{algpseudocode}
\usepackage{subcaption}

\newtheorem{Theorem}{Theorem}[section]

\newtheorem{Conjecture}[Theorem]{Conjecture}

\newtheorem{Lemma}[Theorem]{Lemma}

\title{A stronger upper bound on the D-chromatic index}
\author{Lin Tian\thanks{Corresponding author. Department of Mathematical Sciences, Middle Tennessee State University, Murfreesboro, TN 37132, USA. Email: \texttt{lt5b@mtmail.mtsu.edu}}
\and Runze Wang\thanks{Department of Mathematics and Computer Science, Augustana College, Rock Island, IL 61201, USA. Email: \texttt{runze.w@hotmail.com}}}
\date{}

\begin{document}

\maketitle
\begin{abstract}
    For a graph $G$, a proper edge coloring of $G$ is called a D-coloring if every diamond subgraph of $G$ is rainbow. Let $\chi'_D(G)$ be the D-chromatic index of $G$, which is the smallest integer $k$ such that $G$ admits a D-coloring with $k$ colors. Let $\Delta$ be the maximum degree of $G$. The only known Brooks-type upper bound on $\chi'_D(G)$ is $\frac{9}{16}\Delta^2 + \frac{1}{2}\Delta$, given by a greedy coloring. In this paper, using a probabilistic method, we obtain the first improvement upon this upper bound by proving that $\chi'_D(G) \le (1-c)\frac{9}{16}\Delta^2$ for some $c > 0$ and sufficiently large $\Delta$.
\end{abstract}

\section{Introduction}

    All graphs considered in this paper are finite and simple. Let $G$ be a graph with vertex set $V(G)$ and edge set $E(G)$. For $U\subseteq V(G)$, let $G[U]$ denote the subgraph of $G$ induced by $U$. For $v\in V(G)$, let $N_G(v)$ denote the open neighborhood of $v$. We call $G[N_G(v)]$ the \emph{neighborhood graph} of $v$. For two edges $e_1, e_2\in E(G)$, the distance between $e_1$ and $e_2$ is the distance between the corresponding vertices in the line graph of $G$. (Note that the distance between a pair of incident edges is~$1$.) Let $\Delta(G)$ denote the maximum degree of $G$. When the underlying graph is clear, we may simply write $\Delta$ for $\Delta(G)$.

    An edge coloring of $G$ is called a \emph{proper edge coloring} if incident edges get distinct colors. An edge coloring of $G$ is called a \emph{strong edge coloring} if any two edges with distance at most $2$ get distinct colors. The \emph{strong chromatic index} of $G$, denoted by $\chi'_s(G)$, is the smallest integer $k$ such that $G$ admits a strong edge coloring with $k$ colors. The following well-known conjecture was proposed by Erd\H os and Ne\v set\v ril~\cite{EN}.

    \begin{Conjecture}[Erd\H os and Ne\v set\v ril \cite{EN}]\label{erdos}
        For every graph $G$ with maximum degree $\Delta$,
        \[
            \chi'_s(G)\le\begin{cases}
            \frac{5}{4}\Delta^2 & \mbox{if }\Delta \mbox{ is even}; \\
            \frac{5}{4}\Delta^2-\frac{1}{2}\Delta+\frac{1}{4} & \mbox{if }\Delta \mbox{ is odd}.
            \end{cases}
        \]
    \end{Conjecture}

    While this conjecture remains open, it was proven that $\chi'_s(G)\le 1.998\Delta^2$ for sufficiently large $\Delta$ by Molloy and Reed \cite{MR-97}. Subsequently, this upper bound has been improved to $\chi'_s(G)\le 1.93\Delta^2$ by Bruhn and Joos~\cite{BJ}, to $\chi'_s(G)\le 1.835\Delta^2$ by Bonamy et al.~\cite{BPP}, and to $\chi'_s(G)\le 1.772\Delta^2$ by Hurley et al.~\cite{HDK}.
    
    Under an edge coloring of $G$, a subgraph $H \subseteq G$ is \emph{rainbow} if all the edges in $H$ receive distinct colors. A \emph{B-coloring} of $G$ is a proper edge coloring such that every $4$-cycle in $G$ is rainbow. B-coloring was introduced by Gy\'arf\'as and S\'ark\"ozy~\cite{GS-23} as a "less strong" edge coloring. It has recently been studied in \cite{GMRS,GSW,KWZ}.
    
    A \emph{diamond} is a graph obtained by deleting an edge from $K_4$. Motivated by the notions of strong edge coloring and B-coloring, Wang \cite{Wang-26} introduced a new graph edge coloring called \emph{D-coloring}. For a graph $G$, a D-coloring of $G$ is a proper edge coloring such that every diamond subgraph is rainbow. It is clear that a strong edge coloring is always a B-coloring, and a B-coloring is always a D-coloring.

    The \emph{D-chromatic index} of $G$, denoted by $\chi'_D(G)$, is the smallest integer $k$ such that $G$ admits a D-coloring with $k$ colors. By a greedy coloring, Wang~\cite{Wang-26} proved the following upper bound.
    
    \begin{Theorem}[Wang~\cite{Wang-26}]\label{theorem}
        For every graph $G$ with maximum degree $\Delta$,
        \[\chi'_D(G) \le \frac{9}{16}\Delta^2 + \frac{1}{2}\Delta. \]
    \end{Theorem}
    
    In the same paper, the following Erd\H os-Ne\v set\v ril-type conjecture (in analogy with Conjecture \ref{erdos}) was proposed.
    
    \begin{Conjecture}[Wang~\cite{Wang-26}]\label{conj1}
        For every graph $G$ with maximum degree $\Delta$,
        \[\chi'_D(G) \le \frac{1}{2}\Delta^2 + \frac{1}{2}\Delta. \]
    \end{Conjecture}
    
    If true, this upper bound will be sharp because of the complete graph $K_{\Delta+1}$. This conjecture has been verified for $\Delta \le 5$ in~\cite{Wang-26}, and still remains open for $\Delta \ge 6$. 
    
    For planar graphs, Wang~\cite{Wang-26} proposed the following conjecture, which has recently been verified for $\Delta = 4$, $\Delta = 5$, and $\Delta \ge 33$ by Hu et al.~\cite{Hu-2026}, and remains open for $6 \le \Delta \le 32$.
    
    \begin{Conjecture}[Wang~\cite{Wang-26}]
        For every planar graph $G$ with maximum degree $\Delta\ge 4$,
        \[ 
	       \chi'_D(G)\le  \left\{ \begin{array}{ll}
		     9 &
		   \mbox{ if } \Delta = 4;
		   \\
		   10 & \mbox{ if } \Delta = 5;
          \\
          2\Delta - 1 &
          \mbox{ if } \Delta\ge 6.
	\end{array} \right.
	\]
    \end{Conjecture}

    In this paper, using a probabilistic method developed by Molloy and Reed \cite{MR-97}, we prove a stronger upper bound on $\chi'_D(G)$. More precisely, we improve the coefficient of $\Delta^2$ in Theorem~\ref{theorem} from $\frac{9}{16}$ to $(1 - 0.00016)\cdot\frac{9}{16}$.

    \begin{Theorem}\label{main}
        For every graph $G$ with sufficiently large maximum degree $\Delta$,
        \[\chi'_D(G) \le (1-0.00016)\cdot\frac{9}{16}\Delta^2.\]
    \end{Theorem}

\section{Proof of Theorem \ref{main}}

It suffices to prove Theorem \ref{main} for $\Delta$-regular graphs, since every graph with maximum degree $\Delta$ is a subgraph of a $\Delta$-regular graph, and it is clear that $\chi_D'(K') \le \chi_D'(K)$ if $K'$ is a subgraph of $K$.

To prove Theorem \ref{main}, we will apply the following lemma due to Molloy and Reed \cite{MR-97}.

\begin{Lemma}[Molloy and Reed~\cite{MR-97}]\label{MRlemma}
    Let $\mu, \gamma > 0$ such that $\gamma < \frac{\mu}{2(1-\gamma)}e^{-3/(1-\gamma)}$. Let $H$ be a graph with sufficiently large maximum degree $X$, such that for each $v \in V(H)$, $H[N_H(v)]$ has at most $(1-\mu){X \choose 2}$ edges. Let $\chi(H)$ denote the chromatic number of $H$. Then $\chi(H) \le (1- \gamma)X$.
\end{Lemma}

In order to apply this lemma, for every graph $G$, we will construct an auxiliary graph $H$ such that $\chi(H)=\chi'_D(G)$.

For a graph $G$ and two edges $e$ and $e'$ in $G$, we say that $e$ \emph{sees} $e'$ if 
    \begin{itemize}
        \item $e$ and $e'$ are incident, or
        \item $e$ and $e'$ are not incident but lie in a common diamond subgraph.
    \end{itemize}

Under a D-coloring of $G$, any two edges seeing each other get distinct colors. We define an auxiliary graph $H$ of $G$, where $V(H) = E(G)$, and two vertices $e,e'\in V(H)$ are adjacent if and only if $e$ sees $e'$ in $G$. Note that $\chi(H)=\chi'_D(G)$. In the proof of Theorem \ref{theorem}, Wang \cite{Wang-26} showed that an edge $e\in E(G)$ sees at most $\frac{9}{16}\Delta^2+\frac{1}{2}\Delta-1$ edges, which means a vertex $e\in V(H)$ has at most 
\begin{align}
    \frac{9}{16}\Delta^2+\frac{1}{2}\Delta-1 \label{neighbor}
\end{align}
neighbors.

In the following lemma, we prove that for any $e\in V(H)$, its neighborhood graph $H[N_H(e)]$, which is the subgraph of $H$ induced by the neighbors of $e$ in $H$, cannot be too dense. Note that Lemma \ref{ourlemma} allows us to take $X=\frac{9}{16}\Delta^2 + \frac{1}{2}\Delta - 1$, $\mu = \frac{1}{150}$, and $\gamma = 0.00016$ in Lemma \ref{MRlemma} to complete the proof of Theorem \ref{main}.

\begin{Lemma}\label{ourlemma}
    Let $G$ be a $\Delta$-regular graph with sufficiently large $\Delta$, and let $H$ be the auxiliary graph of $G$ defined as above. Then for each $e \in V(H)$, $H[N_H(e)]$ has at most $ ( 1 - \frac{1}{150} ) {\frac{9}{16}\Delta^2+\frac{1}{2}\Delta-1 \choose 2} $ edges.
\end{Lemma}

\begin{proof}
    Consider an edge $e = uv$ in $G$. Let $X = N_G(u) \backslash \{v\}$, $Y = N_G(v)\backslash \{u\}$, $Z = X \cap Y$, and $W = (X \cup Y) \setminus Z$. Denote $\alpha = |Z|$ and $\beta = |W|$. Since $G$ is $\Delta$-regular, we have $\beta = 2(\Delta - 1 - \alpha) $. Let $m$ be the number of edges in $G$ with both endpoints in $Z$, let $n$ be the number of edges in $G$ with one endpoint in $Z$ and the other endpoint in $W$, and let $s = m + n$. Note that $s$ is the number of edges seeing $e$ but not incident with $e$, and as $e$ is incident with $2\Delta-2$ edges, we know that $e$ sees $s+2\Delta-2$ edges in $G$. So, by the definition of $H$, there are $s+2\Delta-2$ vertices in $H[N_H(e)]$.

    Each edge with both endpoints in $Z$ contributes 2 to the degree sum of the vertices in $Z$, and each edge with one endpoint in $Z$ and the other endpoint in $W$ contributes 1 to the degree sum of the vertices in $Z$, so we have
    \begin{equation}
        2m + n \le \alpha (\Delta - 2). \label{degsum}
    \end{equation}
    This, together with $n \le \alpha \beta = 2\alpha(\Delta - 1 - \alpha)$, yields that 
    \begin{align}
        s = m + n &\le \frac{\alpha(\Delta - 2) - n}{2} + n \nonumber \\
                  &= \frac{\alpha(\Delta - 2) + n}{2} \nonumber \\
                  &\le \frac{\alpha(\Delta - 2) + 2\alpha(\Delta-1-\alpha)}{2} \nonumber \\
                  &= -\alpha^2 + \Bigl(\frac{3}{2}\Delta-2\Bigr)\alpha \nonumber \\
                  &= - \Bigl(\alpha-\Bigl(\frac{3}{4}\Delta - 1\Bigr)\Bigr)^2 + \Bigl(\frac{3}{4}\Delta - 1\Bigr)^2. \label{sbound}
    \end{align}

    Setting $\epsilon = 0.0024$, there are two cases.
    
    \medskip
    \emph{Case 1.} $ s \le (\frac{9}{16} - \epsilon)\Delta^2$. In this case, we have $|H[N_H(e)]| = s+2\Delta - 2 \le (\frac{9}{16} - \epsilon)\Delta^2 + 2\Delta - 2$. For sufficiently large $\Delta$, it holds that $2\Delta - 2 < \frac{\epsilon}{10}\Delta^2$. Hence, $|H[N_H(e)]| < (\frac{9}{16} - \epsilon)\Delta^2 + \frac{\epsilon}{10}\Delta^2 = (\frac{9}{16} - \frac{9\epsilon}{10})\Delta^2$. Therefore, $H[N_H(e)]$ has at most ${(\frac{9}{16} - \frac{9\epsilon}{10})\Delta^2  \choose  2} < (1 - \frac{1}{150}){\frac{9}{16}\Delta^2+\frac{1}{2}\Delta-1 \choose 2}$ edges.

    \medskip
    \emph{Case 2.} $ s > (\frac{9}{16} - \epsilon)\Delta^2$. Combining $ s > (\frac{9}{16} - \epsilon)\Delta^2$ with \eqref{sbound}, we have
    \begin{align*}
        \Bigl(\frac{9}{16} - \epsilon\Bigr)\Delta^2 < - \Bigl(\alpha-\Bigl(\frac{3}{4}\Delta - 1\Bigr)\Bigr)^2 + \Bigl(\frac{3}{4}\Delta - 1\Bigr)^2,
    \end{align*}
    which implies
    \[\Bigl(\alpha-\Bigl(\frac{3}{4}\Delta - 1\Bigr)\Bigr)^2 < \Bigl(\frac{3}{4}\Delta - 1\Bigr)^2 - \Bigl(\frac{9}{16} - \epsilon\Bigr)\Delta^2 = \epsilon\Delta^2 - \frac{3}{2}\Delta + 1 < \epsilon\Delta^2. \]
    Thus $ -\sqrt{\epsilon}\Delta + \frac{3}{4}\Delta - 1 < \alpha < \sqrt{\epsilon}\Delta + \frac{3}{4}\Delta - 1 $. So for sufficiently large $\Delta$, we have
    \[ 0.7\Delta < \alpha <  0.8\Delta \text{\ \ and\ \ } 0.4\Delta  <  \beta = 2(\Delta - 1 -\alpha) < 0.6\Delta. \]
    For $\alpha\in (0.7\Delta,\,0.8\Delta)$, $\alpha\beta$ decreases as $\alpha$ increases, so 
    \begin{align}
        0.32\Delta^2 < \alpha\beta < 0.42\Delta^2. \label{alphabeta}
    \end{align}

    Let $t = \alpha\beta - n$ denote the number of non-edges between $Z$ and $W$. By \eqref{degsum}, we have $2s - n = 2m + n \le \alpha(\Delta - 2)$, which implies $n \ge 2s - \alpha(\Delta - 2)$. So for sufficiently large $\Delta$, we have
    \begin{align}
                t &\le \alpha\beta - 2s + \alpha(\Delta - 2) \nonumber \\
                  &= \alpha \cdot 2(\Delta - 1 - \alpha)  + \alpha(\Delta - 2) - 2s \nonumber \\
                  &= -2\alpha^2 + 3\alpha\Delta - 4\alpha - 2s \nonumber \\
                  &= 2\Bigl(-\alpha^2 + \Bigl(\frac{3}{2}\Delta - 2\Bigr)\alpha\Bigr) - 2s \nonumber \\
                  &< 2\Bigl(\frac{3}{4}\Delta - 1\Bigr)^2 - 2 \Bigl(\frac{9}{16} - \epsilon\Bigr)\Delta^2 \nonumber \\
                  &= 2\epsilon\Delta^2 - 3\Delta + 2 \nonumber \\
                  &<2\epsilon\Delta^2 = 0.0048\Delta^2. \label{tbound}
    \end{align}
    
    We next count the number of non-edges in $G[Z]$ and $G[W]$. Let $p$ be the number of non-edges in $G[Z]$ and $q$ be the number of non-edges in $G[W]$. First, it is evident that $p = {\alpha \choose 2} - m$ and $q = {\beta \choose 2} - \frac{\beta(\Delta-1) - n}{2}$. Then, by \eqref{degsum}, $n = \alpha\beta - t$, $\beta = 2(\Delta - 1 -\alpha)$, \eqref{alphabeta}, and \eqref{tbound}, we have
    \begin{align}
        p = {\alpha \choose 2} - m &\ge \frac{\alpha(\alpha - 1) - \alpha(\Delta - 2)+n}{2} \nonumber \\
                                    &= \frac{\alpha^2 + \alpha - \alpha\Delta + n}{2} \nonumber \\
                                    &= \frac{\alpha(\alpha + 1 - \Delta) + \alpha\beta - t}{2} \nonumber \\
                                    &= \frac{\alpha(-\frac{1}{2}\beta) + \alpha\beta - t}{2} \nonumber \\
                                    &= \frac{\alpha\beta}{4} - \frac{t}{2} \nonumber \\
                                    &\ge 0.0776\Delta^2. \label{pbound}
    \end{align}
    Similarly, by $n = \alpha\beta - t$, $\beta = 2(\Delta - 1 -\alpha)$, $\beta\in(0.4\Delta,0.6\Delta)$, and \eqref{tbound}, we have
    \begin{align}
        q = {\beta \choose 2} - \frac{\beta(\Delta-1) - n}{2} &= \frac{\beta^2 - \beta\Delta + n}{2} \nonumber \\
                                            &= \frac{\beta^2 - \beta\Delta+ \alpha\beta - t}{2} \nonumber \\
                                            &= \frac{\beta^2 + \beta(\alpha-\Delta) - t}{2} \nonumber \\
                                            &= \frac{\beta^2 + \beta(-\frac{1}{2}\beta-1) - t}{2} \nonumber \\
                                            &= \frac{\beta^2}{4} - \frac{\beta}{2} - \frac{t}{2} \nonumber \\
                                            &\ge 0.0376\Delta^2, \label{qbound}
    \end{align}
    when $\Delta$ is sufficiently large.
        
    For two non-edges $z_1z_2,\,w_1w_2$ in $G$ with $z_1,z_2\in Z$ and $w_1,w_2\in W$, we call $(z_1z_2,\,w_1w_2)$ a \emph{good pair} if $\{z_1w_1, z_1w_2, z_2w_1, z_2w_2\}\subseteq E(G)$. If $(z_1z_2,\,w_1w_2)$ is a good pair, then $z_1w_1$ and $z_2w_2$ do not see each other, and $z_1w_2$ and $z_2w_1$ do not see each other. This means, in $H[N_H(e)]$, $z_1w_1$ is not adjacent to $z_2w_2$, and $z_1w_2$ is not adjacent to $z_2w_1$. Thus, each good pair in $G$ gives us two non-edges in $H[N_H(e)]$.

    Now we count the number of good pairs in $G$. First, there are $p$ non-edges in $G[Z]$ and $q$ non-edges in $G[W]$, so we have $pq$ pairs of non-edges in the form $(z_1z_2,\,w_1w_2)$ with $z_1,z_2\in Z$ and $w_1,w_2\in W$. However, for some of these pairs, we may not have $\{z_1w_1, z_1w_2, z_2w_1, z_2w_2\}\subseteq E(G)$. As there are $t$ non-edges between $Z$ and $W$, there are at most $t(\alpha - 1)(\beta - 1) < t\alpha\beta$ pairs of non-edges in the form $(z_1z_2,\,w_1w_2)$ without $\{z_1w_1, z_1w_2, z_2w_1, z_2w_2\}\subseteq E(G)$. Thus, we have at least $pq-t\alpha\beta > 0.0009\Delta^4$ good pairs, by \eqref{alphabeta}, \eqref{tbound}, \eqref{pbound}, and \eqref{qbound}. These good pairs give us $0.0018\Delta^4$ non-edges in $H[N_H(e)]$.
    
    As mentioned in \eqref{neighbor}, we have $|H[N_H(e)]| \le \frac{9}{16}\Delta^2 + \frac{1}{2}\Delta - 1$. Therefore, we conclude that $H[N_H(e)]$ has at most ${\frac{9}{16}\Delta^2 + \frac{1}{2}\Delta - 1 \choose 2} - 0.0018\Delta^4 < ( 1 - \frac{1}{150} ) {\frac{9}{16}\Delta^2+\frac{1}{2}\Delta-1 \choose 2} $ edges, completing the proof of Lemma \ref{ourlemma}.
\end{proof} 

As mentioned earlier, taking $X=\frac{9}{16}\Delta^2 + \frac{1}{2}\Delta - 1$, $\mu = \frac{1}{150}$, and $\gamma = 0.00016$ in Lemma \ref{MRlemma} completes the proof of Theorem \ref{main}.

\section{Remarks}
We made no attempt to optimize $\mu$ and $\gamma$. In fact, if we also consider the edges in $G[Z]$ that do not see each other, then a more careful calculation may improve $\gamma$ to $0.0003$, but the same method appears unlikely to bring the upper bound anywhere close to $\chi'_D(G) \le \frac{1}{2}\Delta^2 + \frac{1}{2}\Delta$, as conjectured in Conjecture \ref{conj1}.

\section*{Statements and Declarations}

\textbf{Funding}: The authors declare that no funds, grants, or other support were received during the preparation of this manuscript.

\noindent\textbf{Data Availability}: Data sharing is not applicable to this article, as no new data were created or analyzed in this study.

\noindent\textbf{Competing Interests}: The authors have no relevant financial or non-financial interests to disclose.
    
\end{document}